\documentclass{article}
\usepackage{amssymb,amsmath,amsthm,graphicx}

\def\qed{\hfill {\hbox{${\vcenter{\vbox{               %HOLLOW SQUARE
   \hrule height 0.4pt\hbox{\vrule width 0.4pt height 6pt
   \kern5pt\vrule width 0.4pt}\hrule height 0.4pt}}}$}}}

\newtheorem{theorem}{Theorem}

\newtheorem{proposition}[theorem]{Proposition}
\newtheorem{corollary}[theorem]{Corollary}

\theoremstyle{definition}
\newtheorem{example}{Example}
\newtheorem{definition}{Definition}
\newtheorem{remark}{Remark}

\date{}

\title{\Large \textbf{The Forbidden Quiver of a Link}}

\author{Sam Nelson\footnote{Email: Sam.Nelson@cmc.edu. Partially supported by Simons Foundation collaboration grant 702597}\and
Stella Shah\footnote{Email: sshah0159@scrippscollege.edu}}

\begin{document}
\maketitle

\begin{abstract}
The forbidden moves in virtual knot theory can be used to unknot any knot, 
virtual or classical; however, multi-component crossings in links can still 
survive, resulting a \textit{fused link}. Using the forbidden moves, we 
categorify fused links obtain a quiver-valued invariant of classical 
and virtual links we call the \textit{forbidden quiver}, opening the way for
functors to and from other categories. As an application we use the forbidden
quiver to obtain three polynomial invariants of virtual and classical links. 
Since these invariants are not sensitive to single-component 
crossing change, they are also link homotopy invariants. 
\end{abstract}

\parbox{5.5in} {\textsc{Keywords:} classical and virtual links, forbidden moves,
quandles, quivers, link invariants, Gauss diagrams, link homotopy, fused links, categorification

\smallskip

}

\section{\large\textbf{Introduction}}\label{I}

In \cite{GPV} the \textit{forbidden moves} on signed Gauss diagrams
representing oriented virtual knots and links were introduced and shown
to be unknotting moves. In \cite{N}, two additional moves were constructed
using the forbidden moves which allow arbitrary movement of arrow heads
and tails of either crossing sign past each other. In particular, 
single-component crossings can be removed with these moves together with
Reidemeister I moves, resulting in unknotting. See also \cite{K,O}.

However, unknotting is not necessarily unlinking -- while some multi-component
crossings can be removed via forbidden moves and Reidemeister II moves, in 
general some multi-component crossings remain. In \cite{FK,NT} etc., virtual
links considered up to forbidden moves are known as \textit{fused links}
and are shown to be classified by their virtual linking numbers.
 In context of knot invariants
defined via quivers such as \cite{CN,N} etc, we characterize the fused link
class of classical and virtual links using a finite quiver with signed
edges. We call the result the \textit{forbidden quiver of the link}. Since
quivers are small categories, this construction provides a categorification of
fused links.
\[\scalebox{0.8}{\includegraphics{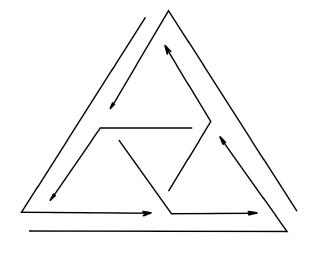}}
\ \quad\includegraphics{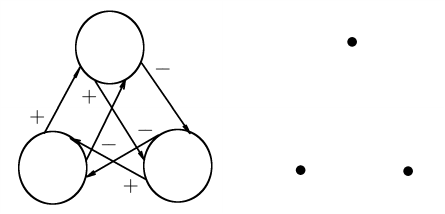}\]
The isomorphism class of this forbidden quiver is an invariant of both ambient 
isotopy and link homotopy and can be used to define several simpler 
invariants as decategorifications. 

The paper is organized as follows. In Section \ref{RB} we review the basics 
of Gauss diagrams and forbidden moves. In Section \ref{BB} we introduce the
forbidden quiver and the associated forbidden polynomial link invariants. We 
prove a result about which quivers can be obtained as the forbidden 
quiver of a virtual link and collect some example computations including
computing the polynomial decategorification polynomials for all classical
links with up to 7 crossings. We conclude in Section \ref{Q} with questions 
for future research.

This paper, including all text, illustrations, and python code for 
computations, was written strictly by the authors without the use of 
generative AI in any form.

\section{\large\textbf{Gauss Diagrams, Forbidden Moves and Quivers}}\label{RB}

In this Section we review the basics needed for the rest of the paper, 
including Gauss diagrams, forbidden moves and quivers. See \cite{CN, EN} 
for more.

Signed Gauss diagrams are combinatorial structures used to represent a knot or 
link diagram through an encoding of its crossing information. The signed 
Gauss code of an oriented knot or link $L$ is determined by labeling each 
crossing with a label and local writhe number
\[\includegraphics{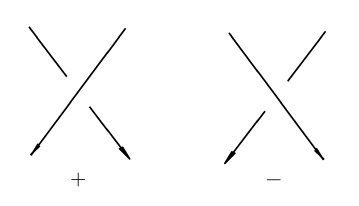}.\]
Then for each individual component of the link, we choose an initial starting 
point, and in the direction of the orientation we list the order in which 
the crossings are encountered and whether they were an over- or 
under-crossing. This list of of lists of crossings with signs is 
a \textit{signed  Gauss code}. 
In a \textit{Gauss diagram}, each link component is represented by a circle, 
labeled counter-clockwise with elements of the Gauss code in order. 
Each crossing on $L$ has a corresponding over- and under-crossing label 
on its Gauss diagram. For each of these crossings, an arrow is drawn from 
the label on the diagram representing the over-crossing to the label 
representing the corresponding under-crossing. Finally, each crossing arrow
is decorated with a ``$+$'' or ``$-$'' to show the sign of the crossing.
The figure eight knot has Gauss diagram
\[\includegraphics{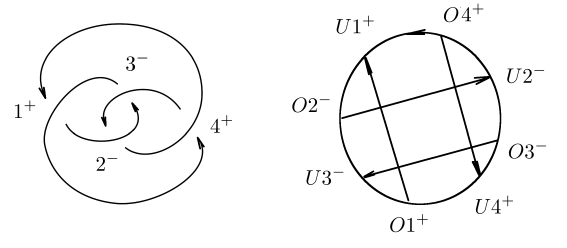}\]
while the Borromean rings have Gauss diagram
\[\includegraphics{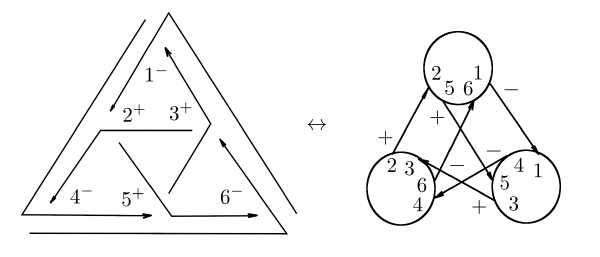}.\]

Given a Gauss diagram for a knot $L$, there is a set of permitted moves 
on the arrows of the Gauss diagram that encode the same knot $L$ up to 
Reidemeister moves. For each of the Reidemeister moves
\[\includegraphics{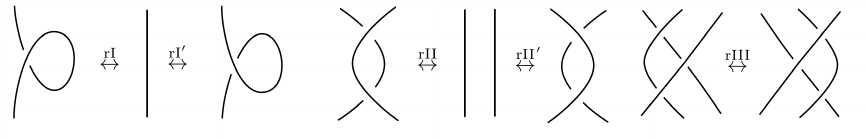}\]
there is a corresponding 
change made to the Gauss diagram reflecting this. The Reidemeister I move R1 
can be understood as the creation or deletion of a kink which does not 
interact with any other arc of the knot. Following along an arc of a knot, 
along an R1 kink is a crossing which has for one arc, no other intersections 
between the over and under crossing. Thus, the Gauss diagram representation 
of R1 consists of deleting or adding an arrow which does not intersect with 
any other arrow. 
\[\includegraphics{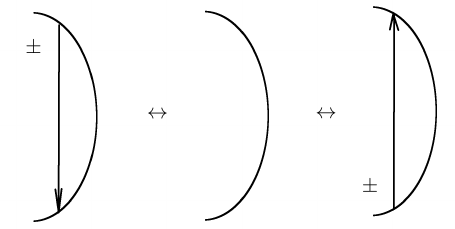}\]

Similarly, the Reidemeister II move R2 describes the action of, given two 
strands, bending one over 
the other, then back to its original relative position. Following along the 
arc of the knot, we find that an R2 move is the creation of two crossings: 
First, one strand crosses over another, then, it crosses back over with 
another crossing of the opposite sign. This move, on the Gauss diagram, 
corresponds to the creation or deletion of two arrows facing the same 
direction, with opposite signs and no other arrow with its head or tail in 
between the heads and the tails of the new arrows. 
\[\includegraphics{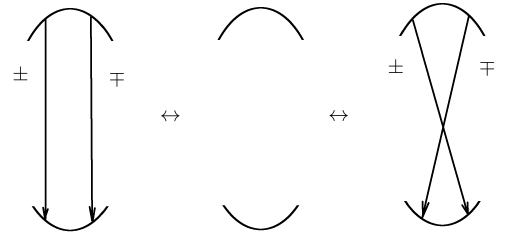}\]

A Reidemeister III move R3 is the action of moving a strand over or under 
an existing crossing between two other strands. This move consists of changing 
the order in which three crossings occur. The Gauss diagram representation of 
this is moving the head and tail of an arrow past the head and tail of two 
non-intersecting arrows with nothing in between their other tail and head. 
\[\includegraphics{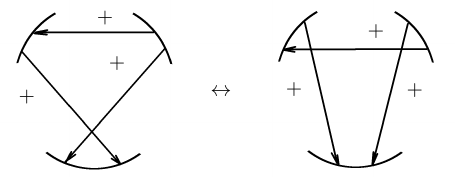}\]

%These three moves, with all their possible oriented cases, show that 
%equivalence classes of 
Gauss diagrams modulo permitted moves are invariants 
of oriented knots and link up to ambient isotopy.
\textit{Virtual knots} are Reidemeister equivalence classes of Gauss diagrams
which might not have planar knot or link diagrams. These are usually drawn
with \textit{virtual crossings} representing genus in the surface on which 
the knot or link diagram is drawn; virtual knots may be regarded as equivalence
classes of knots and links in thickened surfaces modulo stabilization. 
See \cite{EN,K} for more.

In addition, permitted virtual moves vI, vII, vIII and v are analogues of the 
Reidemeister moves using virtual crossings, though they do not alter their 
corresponding Gauss diagram. These moves consist of creating or deleting 
a virtual crossing in a kink, crossing one strand past another through two 
virtual crossings, and moving a crossing, either classical or virtual, past 
a strand through two virtual crossings. 
\[\includegraphics{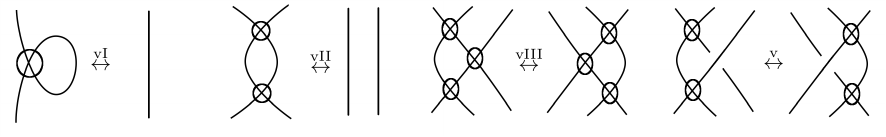}.\]

Gauss diagrams modulo permitted moves and permitted virtual moves are 
invariants of virtual knots up to isotopy. However, these permitted moves 
only describe a portion of potential actions over a Gauss diagram. 

The \textit{forbidden moves} are plausible-looking actions on a Gauss diagram 
which alter the structure of a knot (and hence are forbidden if we wish 
to preserve ambient isotopy). Given a virtual knot K, forbidden moves alter the 
Gauss diagram of K in a manner which does not correspond to permitted 
moves and thus cannot be replicated through a continuous deformation; 
furthermore such moves may alter the status of a knot from classical 
to virtual. 
The first two forbidden moves are moving the head of an arrow past another 
arrowhead
\[\includegraphics{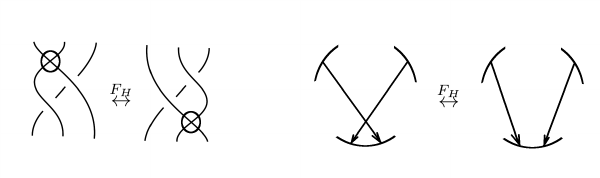}\]
and moving a tail of an arrow past another arrowtail. 
\[\includegraphics{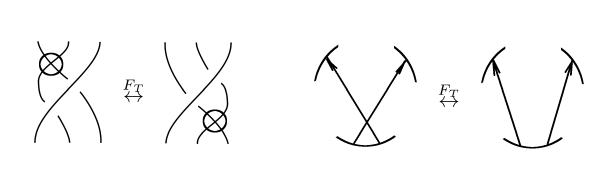}.\]
The other two 
\[\includegraphics{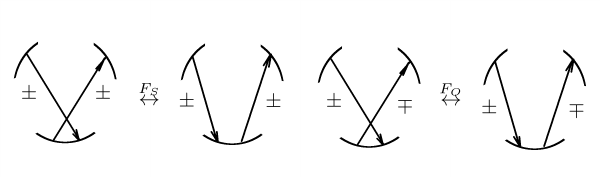}.\]
can be constructed from the first two in the presence of the permitted moves.
In \cite{N} (also see \cite{GPV,K}) it was shown that these forbidden 
moves can completely unknot any knot. However, they may not necessarily 
unlink links.

\section{\large\textbf{Forbidden Quivers}}\label{BB}

In this Section we introduce the forbidden quiver of a link and define new 
polynomial link invariants as an application. We begin by recalling a result 
from \cite{N}.

\begin{theorem} \label{thm1} 
Using both forbidden moves $F_T$ and $F_H$, we can eliminate all 
single-component crossings in a link.
\end{theorem}

Using Theorem \ref{thm1} and Reidemeister II moves, it follows that we can 
reduce any link diagram to include only internally empty circles connected by
sets of parallel arrows of the same sign, possibly in both directions.
For simplicity we will represent $n$ parallel arrows in the same direction
with the same sign as a single arrow with a label of $\pm n$:
\[\includegraphics{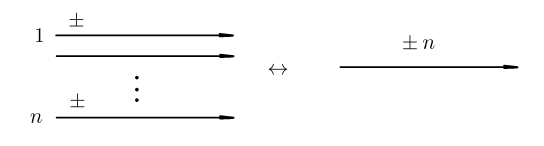}\]

With this in mind, we state our main definition.

\begin{definition}
Let $L$ be a link represented by a Gauss diagram $D$. The \textit{Forbidden
Quiver} of $L$, denoted $\mathcal{FQ}(L)$, is obtained from $D$ by reducing
using forbidden moves and shrinking the circles to vertices. 
\end{definition}

\begin{remark}
The forbidden quiver fundamentally is a quiver with signed arrows. However, 
we can also regard it as a quiver with integer-labeled arrows by replacing
$n$ parallel arrows of the same sign with the same initial and terminal
vertices with a single arrow with the (positive or negative) integer label $n$.
This is well-defined since any pairs of parallel arrows with opposite signs 
have already been canceled in the construction of the forbidden quiver. We
will use both conventions interchangeably.
\end{remark}

\begin{example}
The Borromean rings %u1-o4-u5+o3+, o1-u2+o5+u6-, u3+o2+u4-o6-
have forbidden quiver consisting of three disconnected vertices:
\[\includegraphics{sn-ss-4.pdf}\]
\end{example}

\begin{example}
The Hopf link has forbidden quiver consisting of a single bigon:
\[\includegraphics{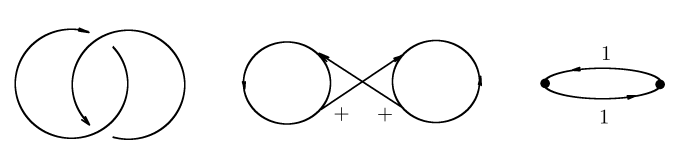}\]
\end{example}

We now come to our main result:

\begin{proposition}
The forbidden quiver of a link is invariant under Reidemeister moves. In 
particular, the isomorphism class of $\mathcal{FQ}(L)$ as a graph with 
signed edges is an invariant of classical and virtual links.
\end{proposition}

\begin{proof}
We must check each of the three types of Gauss diagram Reidemeister moves.
Reidemeister I moves do not change the forbidden quiver since the arrows
they introduce are eliminated in $\mathcal{FQ}(L)$. Likewise, the pairs of 
parallel
arrows of opposite sign introducable by Reidemeister II moves are eliminated
in the construction of $\mathcal{FQ}(L)$, whether within or between circles.
Finally, the effect of a Reidemeister III move is to slide the ends of three
arrows forming a triangle past each other, but this has no effect on the 
quiver since the circles are ultimately reduced to vertices.
\end{proof} 

\begin{remark}
We observe that since the forbidden quiver does not detect 
single-component crossing change, the forbidden quiver and the 
invariants defined from it are all invariants of link homotopy.
\end{remark}

A quiver is a category with vertices as objects and paths as morphisms.
It is then natural to ask what simpler invariants we can obtain from the 
forbidden quiver via decategorification. We begin with a matrix-valued 
invariant and a related integer-valued invariant.

\begin{definition}
Let $L$ be an oriented classical or virtual link of $n$ components
and $\mathcal{FQ}(L)$ its forbidden quiver. Then the \textit{forbidden 
matrix} of $L$ is the matrix $\mathcal{FM}(L)=M_n(\mathbb{Z})$ whose 
entry in row $j$ column $k$ is the label on the arrow from vertex 
$j$ to vertex $k$.
\end{definition}

\begin{proposition}\label{prop:fm}
The forbidden matrix of $L$ is an invariant of classical and virtual links
up to conjugation by a permutation matrix.
\end{proposition}

\begin{proof}
We observe that the entries in $\mathcal{FM}(L)$ are the virtual 
linking numbers $lk_{j/k}$ as defined in \cite{GPV}, i.e., the sum of 
local writhe numbers at crossings where component $j$ crosses over 
component $k$; it follows that the matrix is invariant up to reordering 
of the vertices, i.e., up to conjugation by a permutation matrix. 
\end{proof}

The proof of Proposition \ref{prop:fm} implies the following corollaries:

\begin{corollary}
A classical link must have a symmetric forbidden linking matrix.
\end{corollary}

\begin{proof}
In any classical link, $lk_{j/k}=lk_{k/j}$ for all $j,k$.
\end{proof}

\begin{definition}
The \textit{forbidden determinant} of a link is the determinant of the 
forbidden matrix.
\end{definition}

\begin{corollary}
The forbidden determinant of a link is an invariant of links and an 
invariant of link homotopy.
\end{corollary}

\begin{example}
The Hopf link has forbidden matrix 
\[
\mathcal{FM}(L2a1)=
\left[\begin{array}{rr}
0 & 1 \\
1 & 0
\end{array}\right]
\]
and hence forbidden determinant $-1$.
\end{example}

Next, we define two polynomial invariants of oriented classical and virtual 
links from the forbidden quiver via decategorification.

\begin{definition}
Let $L$ be an oriented classical or virtual link and $\mathcal{FQ}(L)$ its
forbidden quiver. Then the \textit{forbidden polynomial} of $L$ is obtained 
by summing over the edges of the quiver a formal variable $x$ to the power of
the integer weight of the edge
\[\Phi_{\mathcal{F}}(L)=\sum_{e\in E(\mathcal{FQ}(L))} x^{w(e)}.\]
\end{definition}

\begin{definition}
Let $L$ be an oriented classical or virtual link and $\mathcal{FQ}(L)$ its
forbidden quiver. Then the \textit{forbidden 2-variable polynomial} of $L$ 
is obtained by summing over each vertex a product of variables $x$ and $y$ 
to the power of the signed in-degree and signed out-degree of each vertex,
\[\Phi^2_{\mathcal{F}}(L)=\sum_{v\in V(\mathcal{FQ}(L))} 
x^{\mathrm{deg}^+(v)}y^{\mathrm{deg}^-(v)}.\]
\end{definition}

\begin{example}
The Hopf link has forbidden polynomial $\Phi_{\mathcal{F}}(L2a1)=2x$
and forbidden 2-variable polynomial $\Phi^2_{\mathcal{F}}(L2a1)=2xy$.
\end{example}

\begin{example}
The \textit{virtual Hopf link} has forbidden polynomial 
$\Phi_{\mathcal{F}}(2.1)=x$ and forbidden 2-variable polynomial 
$\Phi^2_{\mathcal{F}}(2.1)=x+y$.
\[\includegraphics{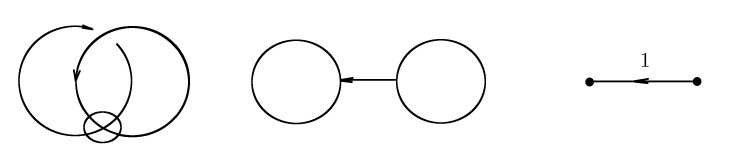}\]
\end{example}

\begin{example}
We computed the forbidden polynomials for each of the prime classical knots
in the Thistlethwaite link table as found at \cite{KA} for a choice of 
orientation for each link. The results are in the table.
\[\begin{array}{r|ll}
L & \Phi_{\mathcal{F}}(L) & \Phi^2_{\mathcal{F}}(L) \\ \hline
L2a1 & 2x & 2xy\\ 
L4a1 & 2x^2 &2x^2y^2\\
L5a1 & 0 & 0\\
L6a1 & 2x^2 & 2x^2y^2\\ 
L6a2 & 2x^3 & 2x^3y^3\\ 
L6a3 & 2x^3 & 2x^3y^3\\ 
L6a4 & 0 & 3 \\ 
L6a5 & 6x & 3x^2y^2\\ 
L6n1 & 6x & 3x^2y^2 \\
\end{array} \quad
\begin{array}{r|ll}
L & \Phi_{\mathcal{F}}(L) & \Phi^2_{\mathcal{F}}(L) \\ \hline
L7a1 & 0 & 0 \\ 
L7a2 & 2x^2 & 2x^2y^2\\ 
L7a3 & 0 & 0 \\ 
L7a4 & 0 & 0 \\ 
L7a5 & 2x & 2xy \\ 
L7a6 & 2x & 2xy \\ 
L7a7 & 4x+2x^{-1} & x^2y^2 + 2\\ 
L7n1 & 2x^{-2} & 2x^{-2}y^{-2} \\
L7n2 & 0 & 0 \\
\end{array}.\]
\end{example}

While our first two decategorifcation polynomials could be reasonably
motivated directly from the virtual linking numbers without our quiver
construction, our next polynomial decategorifcation makes more use of the
quiver structure. 

\begin{definition}
Let $L$ be a classical or virtual link with forbidden quiver 
$\mathcal{FQ}(L)$. We define the \textit{forbidden maximal path polynomial}
of $L$ to the sum over the set of all maximal non-repeating paths $p$ in 
$\mathcal{FQ}(L)$ of terms $x^{\sum_{e\in p}\epsilon(e)}y^{|p|}$ recording the
total sign and path length, i.e.,
\[\Phi_F^{MP}(L)=\sum_{p\in MP}x^{\sum_{e\in p}\epsilon(e)}y^{|p|}.\]
\end{definition}

As previously mentioned, in a quiver considered as a category, paths are the 
morphisms. 
Each path connects two vertices, one at which it starts, and one at which it 
ends. Two paths in a quiver may be composed if the second starts at the 
endpoint of the first path. Non-overlapping paths do not contain any repeated 
subpaths. Maximal non-overlapping paths are paths such that, on the given 
quiver, there is no path that may be composed with it without resulting in a 
repeated subpath. The sum over a certain path is the sum of the coefficients 
of the minimal paths it contains.

\begin{proposition}
The forbidden maximal path polynomial is an invariant of classical isotopy,
virtual isotopy and link homotopy classes of links.
\end{proposition}

\begin{proof}
This follows immediately from the fact that the quiver is invariant.
\end{proof}

\begin{example}
Consider the virtual link $L$ and its forbidden quiver:
\[\includegraphics{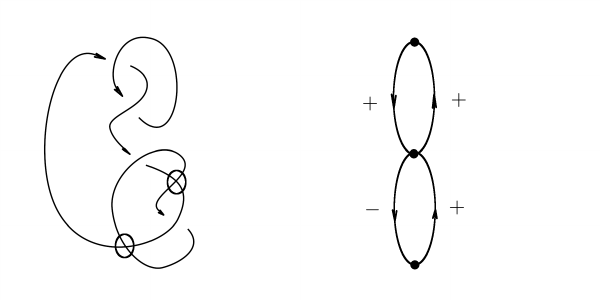}\]
Then there are six maximal non-repeating paths as shown:
\[\includegraphics{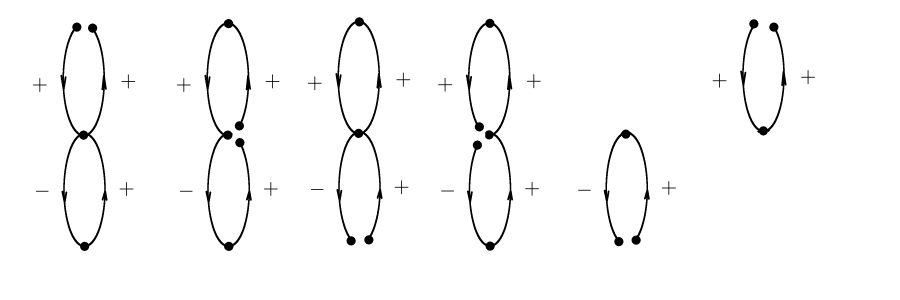}\]
and in particular we obtain forbidden maximal path polynomial
\[\Phi_F^{MP}(L)=4x^2y^4+x^2y^2+y^2.\]
\end{example}

A natural question is to ask which quivers can be obtained as the
forbidden quiver of a link. We first observe that loops are not possible
in a forbidden quiver since they would represent single-component crossings
which get eliminated by Reidemeister I moves. As it turns out, for virtual
links, having no loops is the only necessary condition.

\begin{theorem}
Let $\Gamma$ be any finite signed quiver without loops. Then $\Gamma$ is the 
forbidden quiver of a virtual link $L$.
\end{theorem}

\begin{proof}
Replacing the vertices with circles oriented counterclockwise, we can replace 
the arrows with the virtual tangles as shown
\[\includegraphics{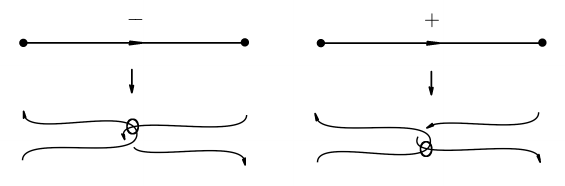}\]
with any arrow crossings replaced with virtual crossings. The
resulting virtual link then has forbidden quiver $\Gamma$.
\end{proof}

\begin{example}
From the quiver below we construct a virtual link which has the given 
quiver as its forbidden quiver.
\[\includegraphics{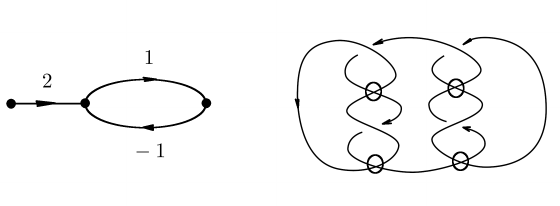}\]
\end{example}

\section{\large\textbf{Questions}}\label{Q}

While every quiver without loops is the forbidden quiver of a virtual link,
it is not as immediately clear which of these virtual links are equivalent 
to classical links.

What other information about the virtual isotopy class and link homotopy class
of a virtual link is contained in the forbidden quiver?

Given any quiver without loops, we obtain a family of virtual links with
the given forbidden quiver. What kinds of functors from the small category 
represented by the quiver into the category of virtual links or other 
categories can be obtained in this way?

\bigskip

\bibliographystyle{abbrv}
\bibliography{sn-ss}

\begin{thebibliography}{1}

\bibitem{KA}
D.~Bar-Natan.
\newblock The knot atlas \textup{http://katlas.org/wiki/Main\_Page}.

\bibitem{CN}
K.~Cho and S.~Nelson.
\newblock Quandle coloring quivers.
\newblock {\em Journal of Knot Theory and Its Ramifications}, 28(01):1950001,
  2019.

\bibitem{EN}
M.~Elhamdadi and S.~Nelson.
\newblock {\em Quandles---an introduction to the algebra of knots}, volume~74
  of {\em Student Mathematical Library}.
\newblock American Mathematical Society, Providence, RI, 2015.

\bibitem{FK}
A.~Fish and E.~Keyman.
\newblock Classifying links under fused isotopy.
\newblock {\em J. Knot Theory Ramifications}, 25(7):1650042, 8, 2016.

\bibitem{GPV}
M.~Goussarov, M.~Polyak, and O.~Viro.
\newblock Finite-type invariants of classical and virtual knots.
\newblock {\em Topology}, 39:1045--1068, 2000.

\bibitem{K}
T.~Kanenobu.
\newblock Forbidden moves unknot a virtual knot.
\newblock {\em J. Knot Theory Ramifications}, 10(1):89--96, 2001.

\bibitem{NT}
T.~Nasybullov.
\newblock Classification of fused links.
\newblock {\em J. Knot Theory Ramifications}, 25(14):1650076, 21, 2016.

\bibitem{N}
S.~Nelson.
\newblock Unknotting virtual knots with {G}auss diagram forbidden moves.
\newblock {\em J. Knot Theory Ramifications}, 10(6):931--935, 2001.

\bibitem{O}
T.~Okabayashi.
\newblock Forbidden moves for virtual links.
\newblock {\em Kobe J. Math.}, 22(1-2):49--63, 2005.

\end{thebibliography}

\noindent
\textsc{Department of Mathematical Sciences \\
Claremont McKenna College \\
850 Columbia Ave. \\
Claremont, CA 91711}

\end{document}